\documentclass[a4paper,reqno]{amsart}

\usepackage{graphicx}
\usepackage{mathrsfs}
\usepackage{amsfonts}
\usepackage{amssymb}
\usepackage{amsmath}
\usepackage{amsthm}
\usepackage{amscd}
\usepackage[all,2cell,poly]{xy}
\usepackage{color}
\usepackage[pagebackref,colorlinks]{hyperref}
\usepackage{enumerate}
\usepackage{bm}
\usepackage{bbm}
\usepackage[normalem]{ulem}
\usepackage{mathrsfs}

\usepackage{tikz}
\usetikzlibrary{arrows.meta}

\usepackage{comment}
\usepackage{hyperref}
\UseAllTwocells \SilentMatrices

\numberwithin{equation}{section}

\usepackage{comment}

\theoremstyle{plain}
\newtheorem{thm}{Theorem}[section]

\newtheorem{cor}[thm]{Corollary}
\newtheorem{lemma}[thm]{Lemma}

\theoremstyle{definition}
\newtheorem{deff}[thm]{Definition}
\newtheorem{example}[thm]{Example}

\theoremstyle{remark}

\newcommand{\K}{\mathsf k}

\newcommand{\GKdim}{\operatorname{GKdim}}

\def\-{\text{-}}

\newcommand{\gr}{\operatorname{gr}}

\newcommand{\id}{\operatorname{id}}

\begin{document}

\title[Graphs with disjoint cycles]{Graphs with disjoint cycles\\ classification via the talented monoid}

\author{Roozbeh Hazrat}
\address{Roozbeh Hazrat: CRMDS, 
Western Sydney University\\
Australia} \email{r.hazrat@westernsydney.edu.au}

\author{Alfilgen N. Sebandal}
\address{Alfilgen N. Sebandal:
Andres Bonifacio Avenue, Tibanga\\
9200 Iligan City, Philippines} \email{alfilgen.sebandal@g.msuiit.edu.ph}

\author{Jocelyn P. Vilela}
\address{Jocelyn P. Vilela:
Andres Bonifacio Avenue, Tibanga\\
9200 Iligan City, Philippines} \email{jocelyn.vilela@g.msuiit.edu.ph}

\begin{abstract} 
We characterise directed graphs consisting of disjoint cycles via their talented monoids. We show that a graph $E$ consists of disjoint cycles precisely  when its  talented monoid $T_E$ has a certain Jordan-H\"older composition series. These are graphs whose associated Leavitt path algebras  have finite Gelfand-Kirillov dimension.  We show that this dimension can be determined as the length of certain ideal series of the talented monoid. Since $T_E$ is the positive cone of the graded Grothendieck group $K_0^{\gr}(L_\K (E))$, we conclude that for graphs $E$ and $F$, if $K_0^{\gr}(L_\K (E))\cong K_0^{\gr}(L_\K (F))$ then  $\GKdim L_\K(E) = \GKdim L_\K(F)$, thus providing more evidence for the Graded Classification Conjecture for Leavitt path algebras. 

\end{abstract}

\maketitle



\section{Introduction}

To a row-finite directed graph $E$, with vertices $E^0$ and edges $E^1$, one can naturally associate a commutative monoid $M_E$, called the \emph{graph monoid} of $E$. It is a
free commutative monoid over the vertices,  subject to identifying a vertex with the sum of vertices it is connected to: 
\begin{equation*}
M_E= \Big \langle \, v \in E^0 \, \, \Big \vert \, \,  v= \sum_{v\rightarrow u \in E^1} u \, \Big \rangle, 
\end{equation*}
when $v$ emits edges. 

The graph monoids were introduced by Ara, Moreno and Pardo~(\cite{ara2006}, \cite[\S 6]{lpabook}) in relations with the theory of Leavitt path algebras. It was shown that the group completion of $M_E$ is the Grothendieck group $K_0(L_\K(E))$, where $L_\K(E)$ is the Leavitt path algebra with coefficient in a field $\K$, associated to $E$.

The so-called \emph{talented monoid} $T_E$ of $E$, can be considered as the ``time evolution model'' of the monoid $M_E$. It is defined as 
\begin{equation*}
T_E= \Big \langle \, v(i), v \in E^0, i \in \mathbb Z  \, \,  \Big \vert \, \, v(i)= \sum_{v\rightarrow u\in E^1} u(i+1) \, \Big \rangle,
\end{equation*}
when $v$ emits edges. $T_E$ is equipped with a $\mathbb Z$-action, defined by ${}^n v(i)=v(i+n)$, where $n\in \mathbb Z$.

In this form the talented monoids were introduced in \cite{hazli} and further studied in \cite{Luiz,lia}. It was shown that the talented monoid captures certain geometry of its associated graph, 
and hence the algebraic properties of the corresponding Leavitt path algebra.  For instance, a graph has Condition (L), i.e., any cycle has an exit, if and only if the group $\mathbb Z$ acts freely on $T_E$~\cite{hazli}. Or, the period of a graph, i.e., the greatest common divisor of the lengths of all closed paths based a vertex,  can be described completely via its associated talented monoid~\cite{Luiz}. Using these results one can give finer descriptions of purely infinite simple Leavitt path algebras associated to graphs. 

The Graded Classification Conjecture states that the graded Grothendieck groups can completely classify Leavitt path algebras~(\cite{mathann},
\cite[\S 7.3.4]{lpabook}).  As the group completion of $T_E$ is the graded Grothendieck group $K_0^{\gr}(L_\K(E))$ (see~\cite{arali}),  the Graded Classification Conjecture can be restated to say that the talented monoid $T_E$ is a complete invariant for the class of Leavitt path algebras.

In this note we show that the talented monoid can be used to classify the class of graphs which consist of disjoint cycles. Specifically,   a graph $E$ consists of disjoint cycles precisely  when its  talented monoid $T_E$ has a certain Jordan-H\"older composition series (Theorem~\ref{mainthm}).
These are precisely those graphs that their associated Leavitt path algebras have finite Gelfand-Kirillov (GK) dimension. We show that the GK-dimension of a Leavitt path algebra $L_\K(E)$ is the length of a certain ideal series in $T_E$ (Theorem~\ref{gkdiml}). We conclude that for graphs $E$ and $F$, if $T_E\cong T_F$ as $\mathbb Z$-monoids, then  $\GKdim L_\K(E) = \GKdim L_\K(F)$. In particular, this shows that if there is a $\mathbb Z$-isomorphism of order groups 
$K_0^{\gr}(L_\K(E)) \cong K_0^{\gr}(L_\K(F))$ then $L_\K(E)$ and $L_\K(F)$ have the same GK-dimension.   This is yet another evidence to support the Graded Classification Conjecture.

In contrast to the talented monoid $T_E$,  the graph monoid $M_E$ can't be used to describe the GK-dimension of a Leavitt path algebra $L_\K(E)$. 
As an example, consider the graphs with disjoint cycles: 

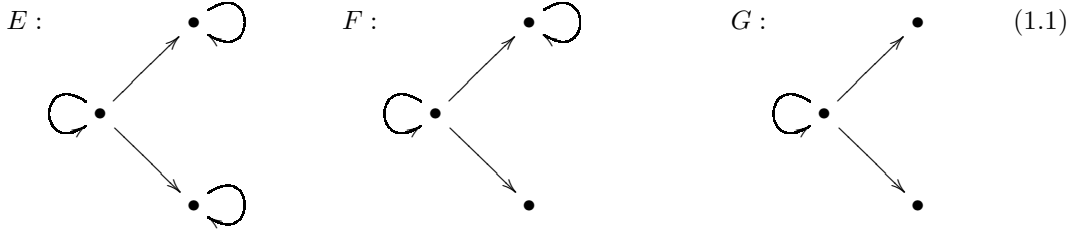
\begin{equation}\label{ggffdd1}
{\def\labelstyle{\displaystyle}
E : \quad \,\,  \xymatrix{
& \bullet \ar@(ur,rd)\\
\bullet  \ar@(lu,ld) \ar[ru]\ar[dr] & \\
& \bullet \ar@(ur,rd)
}}
\qquad \quad \quad \quad
{\def\labelstyle{\displaystyle}
F : \quad \,\,  
\xymatrix{
& \bullet \ar@(ur,rd)\\
\bullet  \ar@(lu,ld) \ar[ru]\ar[dr]  &  \\
& \bullet
}} 
\qquad \quad \quad \quad \quad \quad
{\def\labelstyle{\displaystyle}
G : \quad \,\,  \xymatrix{
& \bullet\\
 \bullet  \ar@(lu,ld)  \ar[ru]\ar[dr]  &   \\
 & \bullet
}} 
\end{equation}

\medskip 

One can easily check that  $M_E\cong M_F \cong M_G$, but we have  $\GKdim L_\K(E)=\GKdim L_\K(F)=3$, whereas $\GKdim L_\K(G)=2$ (see Theorem~\ref{zelgk}).

The paper is organised as follows: In Section~\ref{sec2} we recall the notion of a row-finite directed graph and establish some facts about the hereditary and saturated subsets of a graph. We further discuss the notion of a $\Gamma$-monoid $T$, where $T$ is a monoid and $\Gamma$ is a group acting on $T$. We describe various order-ideals in this setting and the notion of composition series for $T$. We will then describe the Jordon-H\"older theorem in this setting. In Section~\ref{sec3} we tie the notion of graphs and monoids together and define graph monoids and talented monoids. We prove two theorems relating cycles with no exits/sinks in a graph $E$ to the  cyclic/non-comparable minimal order-ideals of the talented monoid $T_E$.  Section~\ref{sec4} contains the main results of the paper. It classifies graphs with disjoint cycles via their talented monoids (Theorem~\ref{mainthm}) and it will use the results of the previous sections to show that for graphs $E$ and $F$, if $T_E\cong T_F$ as $\mathbb Z$-monoids, then $\GKdim L_\K(E) = \GKdim L_\K(F)$ (Corollary~\ref{hfgfhfdd}).

\subsection{Notations}

Throughout we use $\subset$ for the strict inclusion and $\mathbb N$ for the set of non-negative integers. 

\section{Graphs and monoids}\label{sec2}

\subsection{Graphs with disjoint cycles} 

A directed graph $E$ is a tuple $(E^{0}, E^{1}, r, s)$, where $E^{0}$ and $E^{1}$ are
sets and $r,s$ are maps from $E^1$ to $E^0$. A graph $E$ is \emph{finite} if $E^0$ and $E^1$ are both finite. We think of each $e \in E^1$ as an edge 
pointing from $s(e)$ to $r(e)$. A graph $E$ is said to be \emph{row-finite} if for each vertex $v\in E^{0}$,
there are at most finitely many edges in $s^{-1}(v)$. A vertex $v$ for which $s^{-1}(v)$
is empty is called a \emph{sink}.  In a row-finite graph $E$, a vertex  $v\in E^{0}$ which is not a sink,  is called a \emph{regular vertex}. Throughout the paper we only consider row-finite graphs.

We use the convention that a (finite) \emph{path} $p$ in $E$ is
a sequence $p=e_{1} e_{2}\cdots e_{n}$ of edges $e_{i}$ in $E$ such that
$r(e_{i})=s(e_{i+1})$ for $1\leq i\leq n-1$. We define $s(p) = s(e_{1})$, and $r(p) =
r(e_{n})$. We say $n$ is the \emph{length} of $p$. Vertices are considered paths of length zero. A path $p=e_1e_2\dots e_n$ is called a \emph{cycle} if $s(p)=r(p)$ and for any edge $f\in E^1$ with $s(f)=s(e_i)$, for some $1\leq i \leq n$, we have $f=e_i$. 

A \emph{morphism} $\phi:E\rightarrow F$, for directed graphs $E$ and $F$, is a pair of maps (both denoted $\phi$),  $\phi:E^0 \rightarrow F^0$ and $\phi:E^1 \rightarrow F^1$  such that 
$\phi(s(e))=s(\phi(e))$ and $\phi(r(e))=r(\phi(e))$, for any $e\in E^1$. Furthermore, if the maps are injective and $|s^{-1}(v)|=|s^{-1}(\phi(v))|$, for any $v\in E^0$, then $\phi$ is called a \emph{complete morphism}. 

A subset $H \subseteq E^0$ is said to be \emph{hereditary} if
for any $e \in E^1$,  $s(e)\in H$ implies $r(e)\in H$. A subset $H
\subseteq E^0$ is called \emph{saturated} if whenever for a regular vertex $v$, $r(s^{-1}(v)) \subseteq H$ then $v\in H$. Throughout the paper we work with hereditary saturated subsets of $E^0$. If $H$ is a hereditary and saturated subset, we consider the directed graph, denoted by  $H$ again,  consisting of all vertices of $H$ and all edges emitting from these vertices. Clearly there is a complete inclusion morphism $i: H \rightarrow E$. For a subset $S$, we also need to recall the construction of hereditary and saturated subset $\langle S \rangle$ and a fact about them. Let $S$ be a hereditary subset of $E^0$. Then $\langle S \rangle :=\bigcup_{i\in \mathbb N} S_i$, where $S_0=S$ and $S_i=\{v\in E^0 \mid r(s^{-1}(v))\subseteq S_{i-1}\}$ is the smallest hereditary and saturated subset of $E$ containing $S$. This simple lemma will be used throughout the paper.

\begin{lemma} \label{longpath}
Let $E$ be a row-finite graph, $S$ a subset of $E$ and  $\langle S \rangle$  its hereditary and saturated closure. Let $v\in \langle S \rangle$. 
\begin{enumerate}[\upshape(1)]

\item There is path $p$ with $s(p)=v$ and $r(p)\in S$.

\smallskip

\item Any path starting from $v$ eventually ends in $S$, i.e., for any path $p$ with $s(p)=v$, there is a path $q$ such that $s(q)=r(p)$ and $r(q)\in S$. 

\smallskip
\item  There is an $l\in \mathbb N$ such that any path $p$ with $s(p)=v$ of length larger than $l$ ends in $S$, i.e.,  $r(p)\in S$.

\end{enumerate}
\end{lemma}

\begin{proof} The following argument implies all three statements of the lemma. 
Let $v\in \langle S \rangle$.  Then there is an $l\in \mathbb N$ such that $v\in S_l$. Let $p=e_1e_2\dots e_l$ be a path of length $l$ starting from $v$. Since $v\in S_l$, by the construction of $S_l$, we have $r(e_1)=s(e_2) \in S_{l-1}$. Thus $r(e_2)=s(e_3) \in S_{l-2}$. Continuing, $r(p)=r(e_l)\in S_{l-l}=S$. 
\end{proof}

We need the following correspondence between the hereditary and saturated subsets of a graph and its quotient graph. 
For hereditary saturated subsets $H$ and $K$ of the graph $E$ with  $H \subseteq  K$, define the quotient graph $K / H $ as a graph such that 
$(K/ H)^0=K^0\setminus H^0$ and $(K/H)^1=\{e\in E^1\;|\; s(e)\in K, r(e)\notin H\}$. The source and range maps of $K/H$ are restricted from the graph $E$. If $K=E^0$, then $K/H$ is the \emph{quotient graph} $E/H$ (\cite[Definition~2.4.11]{lpabook}). With this construction, 
 $(K/ H)^0$ is a hereditary and saturated subset of the graph $E/H$, which we also denote $K/H$.  We have a natural (inclusion) morphism $\phi:E/H\rightarrow E$ which is injective on vertices and edges.  
 
 The following lemma will be used in a crucial way in Theorem~\ref{gkdiml} to calculate the Gelfand-Kirillov dimension of a Leavitt path algebra from its associated talented monoid.  Denote $\mathcal H(E,H)$ to be all hereditary and saturated subsets of the graph $E$ containing the hereditary and saturated subset $H$ and $\mathcal H(E):=\mathcal H(E, \emptyset)$. 

\begin{lemma} \label{gdeppok}
Let $E$ be a row-finite graph and $H$ a hereditary and saturated subset of $E$. 
There is a one to one correspondence 
\begin{align}\label{gfhfgdgdh3}
\psi:\mathcal H(E,H) &\longrightarrow \mathcal H(E/H),\\
K &\longmapsto K/H\notag
\end{align}
such that $(E/H) \large /(K/H)=E/K$.
\end{lemma}
\begin{proof}
The proof is routine and we leave it to the reader. 
\end{proof}

We refer the reader to \cite{lpabook} for other concepts on graph theory related to path algebras.   The following facts will be used throughout the paper including in the main Theorems~\ref{mainthm} and \ref{gkdiml}. We collect them here for the convenience of the reader.

\begin{lemma} \label{helplemm}
Let $E$ be a row-finite graph, $H$ a hereditary and saturated subset of $E$ and $\phi:E/H \rightarrow E$ the natural morphism of graphs. 

\begin{enumerate}[\upshape(1)]

\item If $E$ has no sinks then the graph $E/H$ has no sink.  

\smallskip

\item If $H$ contains all the sinks, then the graph $E/H$ has no sink. 

\smallskip 

\item Let $H$ be generated by all sinks in $E$.  Then there is one-to-one correspondence between cycles in $E/H$ and $E$ via the morphism $\phi$.

\smallskip

\item  Let $H$ be generated by vertices on cycles with no exits. Then there is one-to-one correspondence between sinks in $E/H$ and $E$ via the morphism $\phi$.

\end{enumerate}
\end{lemma}
\begin{proof}
(1) [and (2)] Suppose $v$ is a sink in $E/H$, which in particular, implies that $v\in E\backslash H$. Since $E$ has no sinks [$H$ contains all the sinks], then there are edges emitting from $v$ in $E$. By construction of $E/H$, in order $v$ to be a sink in $E/H$, all the edges emitting from $v$ should land in $H$. But since $H$ is saturated, this means $v$ has to be in $H$ which is not the case. Thus $E/H$ has no sinks. 

(3) Since the map $\phi:E/H \rightarrow E$ is an inclusion, any cycle appearing in $E/H$, does appear in $E$ as well. Suppose $c=e_1\dots e_k$ is a cycle in $E$ based on the vertex $v$. We show that $v \not \in H$.  If $v\in H$,  then by Lemma~\ref{longpath}, any long enough path starting from $v$ ends in a sink (the generating set of $H$). Thus considering the path  $c^l$, where $l\in \mathbb N$ is large enough, implies that $v$ is a sink which is not the case. Thus $v$ is not in $H$ and consequently all the vertices and edges of $c$ are not in $H$. Therefore the cycle $c$ appears in $E/H$. 

(4) If $v$ is a sink in $E$, then $v\not \in H$. Otherwise by Lemma~\ref{longpath},  there is path from $v$ which connects to a cycle, a contradiction to $v$ being a sink. Thus any sink in $E$ appears as a sink in $E/H$. Furthermore, if $w \in E\backslash H$ is not a sink in $E$, it can't be a sink in $E/H$, otherwise the range of all the edges emitting from $w$ has to be in $H$, i.e., $w\in H$ which is not the case. 
\end{proof}

In this note we will work with graphs which consists of disjoint cycles (we don't necessarily assume the graphs are connected). A directed graph  that every vertex  is the base of at most one cycle is called \emph{a graph with disjoint cycles}. Clearly in such a graph any two distinct cycles do not have a common vertex.  Acyclic graphs (i.e., graphs with no cycles) are examples of graphs with disjoint cycles.   For two distinct cycles $C_1$, $C_2$ , we write $C_1 \Rightarrow C_2$, if there exists a path that starts in $C_1$ and ends in $C_2$. Figure~\ref{ggffdd1} shows graphs with disjoint cycles. 

Among graphs with disjoint cycles, we distinguish two classes: We say a graph with disjoint cycles is a \emph{multiheaded-comet graph} if no cycles has an exit. We say such a graph is a \emph{comet graph} if it has only one cycle. 

For our induction steps later in the paper, we need the following Lemma  regarding the chain of cycles in a graph with disjoint cycles. Although one can give a purely graph theoretical argument, we prove the lemma using the talented monoid in Section~\ref{talmon} to exhibit the techniques used in the paper. 

\begin{lemma}\label{hnhngff}
Let $E$ be a finite graph with disjoint cycles and no sinks. Let $H$ be the hereditary and saturated subset generated by all cycles with no exits in $E$. Then the quotient $E/H$ is a graph with disjoint cycles and no sinks. Furthermore, if $C_1\Rightarrow C_2\Rightarrow \cdots \Rightarrow C_{i-1} \Rightarrow C_{i}$ is a maximal chain of cycles in $E$ then  $C_1\Rightarrow C_2\Rightarrow  \cdots \Rightarrow C_{i-1}$ is a maximal chain of cycles in $E/H$. 
\end{lemma}

For a (row-finite) graph $E$, the notion of Leavitt path algebras has received substantial attentions. We denote a Leavitt path algebra with coefficient in a field $\K$, by $L_\K(E)$. We refer the reader to the book \cite{lpabook} for all the background in the subject. We recall the main theorem of \cite{zel}, which characterises when a Leavitt path algebra has a finite Gelfand-Kirillov dimension.

\begin{thm}[Alahmadi, Alsuhami, Jain, Zelmanov~\cite{zel}]\label{zelgk}
Let $E$ be a finite graph.

\begin{enumerate}[\upshape(i)]
\item  The Leavitt path algebra $L_\K(E)$ has polynomially bounded growth if and only if $E$ is a graph with disjoint cycles. 

\smallskip

\item  If $d_1$ is the maximal length of a chain of cycles in $E$, and $d_2$ is the maximal length of chain of cycles with an exit, then $\GKdim L_\K(E) = \max(2d_1-1, 2d_2)$.
\end{enumerate}
\end{thm}

We can easily see that one can replace $d_2$ by the maximal length of chain of cycles which ends in a sink (thus $d_2$ can be zero here) and the formula can be written as 
\begin{equation}\label{gkformula11}
\GKdim L_\K(E)= \begin{cases}
2d_1 &\text{if $d_1=d_2$;}\\
2d_1-1 &\text{if $d_1 \not = d_2$.}
\end{cases}
\end{equation}

We note that the Leavitt path algebras of these graphs with regard to their irreducible representations were further investigated by Ara and Rangaswamy \cite{araranga,ranga}. 

\subsection{$\Gamma$-monoids}

Let $M$ be a commutative (abelian) monoid.  We define the \emph{algebraic}  pre-ordering on the monoid $M$ by $a\leq b$ if $b=a+c$, for some $c\in M$. 
Throughout we write $a \parallel b$ if the elements $a$ and $b$ are not comparable.  A commutative monoid $M$ is called \emph{conical} if $a+b=0$ implies $a=b=0$ and it is called \emph{cancellative} if $a+b=a+c$ implies $b=c$, where $a,b,c\in M$. The monoid $M$ is called \emph{refinement} if $a+b=c+d$ then there are $e_1,e_2,e_3,e_4\in M$ such that $a=e_1+e_2$, $b=e_3+e_4$ and $c=e_1+e_3$, $d=e_2+e_4$. 
An element $0 \not = a\in M$ is called an \emph{atom} if $a=b+c$ then $b=0$ or $c=0$. An element $0 \not=a\in M$ is called \emph{minimal} if $0 \not = b\leq a$ then $a\leq b$. When $M$ is conical and cancellative, these notions coincide with the more intuitive definition of minimality, i.e., $a$ is minimal if $0\not = b\leq a$ then $a=b$. The monoids of interest in the paper, the talented monoids $T_E$ of directed graphs $E$, are conical, cancellative and refinement and thus all these concepts coincide~\cite{arali}. A submonoid $I$ of $M$ is called \emph{normal} if $a,a+b\in I$ implies $b\in I$. Furthermore $I$ is called an \emph{order-ideal} if $a+b\in I$ implies $a,b\in I$. We refer the reader to the book of Wehrung~\cite{wehrung} for a comprehensive treatment of such monoids and their relations with several branches of mathematics.

Let $T$ be a commutative monoid with a group $\Gamma$ acting on it.  (Throughout we use the letter $T$ to denote a monoid with $\Gamma$-actions.)  For $\alpha \in \Gamma$ and $a\in T$, we denote the action of $\alpha$ on $a$ by ${}^\alpha a$. 
A monoid homomorphism $\phi:T_1 \rightarrow T_2$ is called $\Gamma$-\emph{module homomorphism} if $\phi$ respects the action of $\Gamma$, i.e., $
\phi({}^\alpha a)={}^\alpha \phi(a)$.  A $\Gamma$-\emph{normal} (\emph{order-ideal}) of a $\Gamma$-monoid $T$ is a  normal (order-ideal) $I$ of $T$  which is closed under the action of $\Gamma$.  
We say $T$ is a \emph{simple} 
$\Gamma$-\emph{monoid} if the only $\Gamma$-order-ideals of $T$ are $0$ and $T$.

Throughout we assume that the group $\Gamma$ is abelian. Indeed in our setting of graph algebras, this group is the group of integers $\mathbb Z$. We use the following terminologies throughout this paper: For a $\Gamma$-monoid $T$, we say an element $a\in T$ is \emph{periodic} if there is an $0 \not = \alpha \in \Gamma$ such that ${}^\alpha a =a$. If $a\in T$ is not periodic, we call it \emph{aperiodic}. We denote the orbit of the action of $\Gamma$ on an element $a$ by $O(a)$, so $O(a)=\{{}^\alpha a \mid \alpha \in \Gamma \}$.

For $a\in T$, we denote the  $\Gamma$-order-ideal generated by the element $a$ by $\langle a \rangle $. It is easy to see that 
\begin{equation}\label{oridealhj}
\langle a \rangle=\Big \{ x \in M \mid x \leq \sum_{i\in \Gamma} {}^\alpha a \Big \}. 
\end{equation}

In the setting of $\Gamma$-monoids, several types of ideals can be considered. These ideals will have specific description in the setting of graph monoids.

\begin{deff}\label{noncomdd}
Let $T$ be a $\Gamma$-monoid and $I$ an $\Gamma$-order-ideal of $T$. We say 
\begin{enumerate}[\upshape(i)]
\item $I$ is a \emph{cyclic ideal} if for any $x\in I$, there is an $0\not= \alpha \in \Gamma$ such that ${}^\alpha x=x$;

\item $I$ is a \emph{comparable ideal} if for any $x\in I$, there is an $ 0\not= \alpha \in \Gamma$ such that ${}^\alpha x \geq  x$;

\item $I$ is a \emph{non-comparable ideal} if for any $0\not = x\in I$, and any $0\not=\alpha \in \Gamma$, we have  ${}^\alpha x \parallel x$.

\end{enumerate}

\begin{example}
Let $T=\mathbb N \oplus \mathbb N \oplus \mathbb N \oplus \mathbb N$ be a free abelian monoid with the action of $\mathbb Z$ on $T$ defined by 
${}^1 (a,b,c,d)= (d,a,b,c)$ and extended to $\mathbb Z$. Then $T$ is a cyclic monoid as for any $x\in T$ we have ${}^4x=x$. In fact $T$ is a simple $\mathbb Z$-monoid and one can show that $T$ coincides with the talented monoid $T_E$ of the graph $E$ of  a cycle of length $4$. 
\end{example}

Let $I$ be a submonoid of a monoid $M$. Define an equivalence relation $\sim_{I}$ on $M$ as follows: For $a, b\in M$,  $a\sim_{I} b$ if there exist $i,j\in I$ such that  $a+i=b+j$ in $M$. This is a congruence relation  and thus one can form the quotient monoid   $M/\sim$ which we will denote by $M/I$. 
 The notion of $a\sim_{I} b$ is equivalent to $(a+I) \cap (b+I) \not = \emptyset$. Observe that $a\sim_{I} 0$ in $M$ for any $a\in I$. If $I$ is a normal submonoid then $a\sim_I 0$ if and only if $a\in I$. 
 Furthermore if  $M$ and $I$ are equipped with a $\Gamma$-action then $M/\sim$ is a $\Gamma$-monoid. If $T$ is a \emph{$\Gamma$-refinement monoid} (i.e., a $\Gamma$-monoid which is refinement) and $I$ and $J$ are $\Gamma$-order ideals, one can check that $I+J$ is a $\Gamma$-order-ideal and $T/I$ is a $\Gamma$-refinement monoid.

One of the main result of this note (Theorem~\ref{mainthm}) shows that one can characterise certain type of graphs and their associated Leavitt path algebras with their talented monoids, if these monoids have certain composition series. For the next definition, recall that we use $\subset$ for a strict inclusion. Note that if $J$ is a $\Gamma$-order-ideal of the $\Gamma$-monoid $T$ and $I$ is a $\Gamma$-order-ideal of $J$, then $I$ is a $\Gamma$-order-ideal of $T$.

\end{deff}
\begin{deff}
Let $T$ be a $\Gamma$-monoid. We say that $T$ has a \emph{composition series} if there is a chain of $\Gamma$-order-ideals 
\[ 0=I_0 \subset I_1 \subset I_{2} \subset \dots \subset I_n=T,\] such that $I_{i+1}/I_{i}$, $0 \leq i \leq n-1$, are simple $\Gamma$-monoids. We say a composition series is of the \emph{cyclic (non-comparable, comparable) type} if all the simple quotients  $I_{i+1}/I_{i}$ are cyclic (non-comparable, comparable). We further say, a composition series is of \emph{mixed type} of certain kinds if the simple quotients are of those given kinds. We say two composition series of $T$ are \emph{equivalent} if there is a one-to-one correspondence between the simple quotients of the series such that the corresponding quotients are $\Gamma$-isomorphic monoids. 
\end{deff}

Since the sum of cyclic ideals is cyclic, any $\Gamma$-monoid $T$  has a largest cyclic ideal. We also need to consider the following sequence of ideals in $T$. 

\begin{deff}\label{gdhfyrhfu}
Let $T$ be a $\Gamma$-monoid. The \emph{upper cyclic series}  of $T$ is a chain of $\Gamma$-order-ideals 
\[ 0=I_0 \subset I_1 \subset I_{2} \subset \dots \subset I_n,\]
where $I_{i+1}/I_{i}$ is the largest cyclic ideal of $T/I_i$, $0 \leq i \leq n-1$. 
 We call $I_n$ the \emph{leading ideal} of the series and denote $n$ by $l_c(T)$. 
\end{deff}

As we will see in \S \ref{talmon}, for a graph $E$, the talented monoid $T_E$ has a non-zero cyclic ideal if and only if $E$ has a cycle with no exit (see Theorem~\ref{hfghfbggf}).  Thus $l_c(T_E)=0$ if and only if $E$ has no cycle without exit. Furthermore,  $l_c(T_E)=1$ if and only if $T_E$ is a multi-headed comet graph (Corollary~\ref{gatzia}).

In Theorem~\ref{mainthm} we will prove that  the talented monoid of a graph with disjoint cycles has a composition series of cyclic and non-comparable mixed type. The Jordan-H\"older Theorem for such monoids guarantees that any composition series of such graphs has the same mixed types.

The relations between quotient monoids and homomorphisms follow the same pattern as groups and rings. 
We need the following Lemmas which hold when $T$ is a $\Gamma$-refinement monoid.

\begin{lemma}\label{secondiso}
Let $T$ be a  $\Gamma$-refinement monoid and let $I$ and $J$ be $\Gamma$-order-ideals of $T$. Then we have a $\Gamma$-isomorphism 
$(I+J)/I \cong J/ I\cap J$. \end{lemma}
\begin{proof}
Define the map $\theta:(I+J)/I \rightarrow J/ I\cap J; [a+b]\mapsto [b]$, where $a\in I$ and $b\in J$. Once we show that this map is well-defined, it is then easy to see $\theta$ is an isomorphism of $\Gamma$-monoids. Since $[a+b]=[b]$ in $(I+J)/I$, it is then enough to show $b_1\sim_I b_2$ implies $b_1 \sim_{I\cap J} b_2$. From $b_1\sim_I b_2$, we obtain $b_1+i_1=b_2+i_2$, for $i_1,i_2\in I$. Since $T$ is refinement, we can write $b_1=e_1+e_2$, $b_2=e_1+e_3$, and $i_1=e_3+e_4$ and $i_2=e_2+e_4$, where $e_1,e_2,e_3,e_4 \in T$. Since $I$ and $J$ are $\Gamma$-order-ideals it follows that $e_1\in J$ and $e_2,e_3\in I\cap J$. It now follows that $b_1 \sim_{I\cap J} b_2$. 
\end{proof}

If $T$ is a $\Gamma$-monoid and $I$ an order-ideal of $T$, then there is a one-to-one inclusion-preserving correspondence between $\Gamma$-order-ideals of the quotient monoid $T/I$ and $\Gamma$-order-ideals of $T$ containing $I$. For $I\subseteq J$, the corresponding $\Gamma$-order-ideal in $T/I$ is denoted $J/I$. And, the so-called third isomorphism theorem also holds in this setting: If $I\subset J$ are $\Gamma$-order-ideals of $T$, then 
\begin{equation}\label{thirdiso}
(T/I) \big / (J/I) \cong T/J.
\end{equation}

The isomorphism theorems~\ref{secondiso} and~\ref{thirdiso}  coupled with the correspondence between ideals give the following lemma which will be used later. 
\begin{lemma}\label{compso}
Let $T$ be a $\Gamma$-refinement monoid and $I$ a $\Gamma$-order-ideal of $T$. Then $T$ has composition series if and only if $T/I$ and $I$ have composition series. 
\end{lemma}

The next lemma will be used in Section~\ref{sec4}, to obtain composition series for graphs with disjoint cycles. 

\begin{lemma}\label{rtfdm}
Let $I_1,I_2, \dots, I_k$ be distinct minimal $\Gamma$-order-ideals of a $\Gamma$-refinement monoid $T$. Then 
\[ 0 \subset I_1 \subset I_1+ I_2 \subset \dots \subset I_1+I_2+\dots +I_k\] is a composition series for the monoid $I_1+I_2+\dots +I_k$. 
\end{lemma}
\begin{proof} 
Since $I_i$ are distinct $\Gamma$-order-ideals, it is easy to show that the chain is proper. 
For $1< i \leq k$, set $I=I_1+I_2+\dots +I_{i-1}$ and $J=I_i$. We have $I \cap J=0$. Lemma~\ref{secondiso} now gives
\begin{equation}\label{hfgbgdgd}
\frac{I_1+I_2+\dots +I_{i-1}+I_i}{ I_1+I_2+\dots +I_{i-1}} \cong \frac{I+J}{I}\cong J\cong I_i.
\end{equation}
 Thus~(\ref{hfgbgdgd}) is an isomorphism of monoids, implying for any $1< i \leq k$, the quotients are simple. This proves the Lemma. 
\end{proof}

Finally we describe the Jordan-H\"older theorem in the setting of $\Gamma$-monoids. Although, we will not be using this theorem in this paper, the Jordan-H\"older theorem guarantees that simple monoids associated to the talented monoid of a graph with disjoint cycles are unique. 

\begin{thm}[Jordan-H\"older Theorem] \label{jordan} Any two composition series of a $\Gamma$-refinement monoid $T$ are equivalent. Thus any $\Gamma$-refinement monoid having a composition series determines a unique list of simple $\Gamma$-monoids. 
\end{thm}

Since we have all the isomorphism theorems in the setting of monoids, the proof of Theorem~\ref{jordan} is analogous of the one in group theory. We refer the reader to \cite{alfi}, where this theorem for $\Gamma$-monoids is established adopting Baumslag's short proof in the group setting.

\section{The talented monoid of a graph}\label{sec3}

\subsection{Talented monoid $T_E$ of a directed graph $E$}\label{talmon}

In this section we define the graph monoids that are the main interests of this paper.  Given a row-finite graph $E$, we denote by $F_E$ the free commutative monoid generated by $E^0$.

\begin{deff}\label{def:graphmonoid}
    Let $E$ be a row-finite graph. The \emph{graph monoid} of $E$, denoted $M_E$, is the commutative 
monoid generated by $\{v \mid v\in E^0\}$, subject to
\[v=\sum_{e\in s^{-1}(v)}r(e),\]
for every $v\in E^0$ that is not a sink.
\end{deff}

The relations defining $M_E$ can be described more concretely: First, define a relation $\rightarrow_1$ on $F_E$ as follows: for $\sum_{i=1}^n v_i  \in F_E$, and a regular vertex $v_j\in E^0$, where $1\leq j \leq n$,  
\[\sum_{i=1}^n v_i \rightarrow_1 \sum_{i\not = j }^n v_i+  \sum_{e\in s^{-1}(v_j)}r(e).\]
Then $M_E$ is the quotient of $F_E$ by the congruence generated by $\rightarrow_1$.

The following lemma is essential to the remainder of this paper, as it allows us to translate the relations in the definition of $M_E$ in terms of the simpler relation $\rightarrow$ in $F_E$. Here $\rightarrow $ is the transitive and reflexive closure of $\rightarrow_{1}$, namely $a \rightarrow b$ if there is a finite sequence $a=a_0 \rightarrow_{1} a_1 \rightarrow_{1} \dots \rightarrow_{1} a_n=b$. 

\begin{lemma}[{\cite[Lemmas 4.2 and 4.3]{ara2006}}]\label{confuu}
    Let $E$ be a row-finite graph.
    \begin{enumerate}
        \item (The Confluence Lemma) If $a,b\in F_E\setminus\left\{0\right\}$, then $a=b$ in $M_E$ if and only if there exists $c\in F_E$ such that $a\rightarrow c$ and $b\rightarrow c$. (Note that, in this case, $a=b=c$ in $M_E$.)
        
        \item If $a=a_1+a_2$ and $a\rightarrow b$ in $F_E$, then there exist $b_1,b_2\in F_E$ such that $b=b_1+b_2$, $a_1\rightarrow b_1$ and $a_2\rightarrow b_2$.
    \end{enumerate}
\end{lemma}

The talented monoid of a graph $E$ is the graph monoid of the covering graph of $E$.  The \emph{covering graph} of $E$ is the graph $\overline{E}$ with vertex set $\overline{E}^0=E^0\times\mathbb{Z}$, and edge set $\overline{E}^1=E^1\times\mathbb{Z}$. The range and source maps are given as
\[s(e,i)=(s(e),i),\qquad r(e,i)=(r(e),i+1).\] Note that the graph monoid $M_{\overline E}$ has a natural $\mathbb Z$-action by ${}^n (v,i)= (v,i+n)$. One can directly define the talented monoid of the graph without invoking the notion of the covering graph.

\begin{deff}\label{talentedmon}
Let $E$ be a row-finite directed graph. The \emph{talented monoid} of $E$, denoted $T_E$, is the commutative 
monoid generated by $\{v(i) \mid v\in E^0, i\in \mathbb Z\}$, subject to
\begin{equation}\label{transgfrt}
v(i)=\sum_{e\in s^{-1}(v)}r(e)(i+1),
\end{equation}
for every $i \in \mathbb Z$ and every $v\in E^{0}$ that is not a sink. The group $\mathbb{Z}$ of integers acts on $T_E$ via monoid automorphisms by shifting indices: For each $n,i\in\mathbb{Z}$ and $v\in E^0$, define ${}^n v(i)=v(i+n)$, which extends to an action of $\mathbb{Z}$ on $T_E$. Throughout we will denote elements $v(0)$ in $T_E$ by $v$. 
\end{deff}

For a graph $E$, the graph monoid $M_E$ is conical and refinement~\cite{ara2006}, whereas the talented monoid $T_E$ is conical, cancellative and refinement~\cite{arali}. Clearly there is a $\mathbb Z$-monoid isomorphism $T_E \rightarrow M_{\overline E}; \, v(i) \mapsto (v,i)$. Thus one can use Confluence Lemma \ref{confuu} in the setting of talented monoids (by passing the elements to $M_{\overline E}$). Therefore if $a=b$ in $T_E$, then there is a $c\in F_{\overline E}$ such that $a\rightarrow c$ and $b\rightarrow c$.

\begin{example}\label{perioddel}
Let $E$ be a graph which has a cycle $C$ with no exit. Write $C=e_1e_2\dots e_n$, where $e_i\in E^1$ are edges on the cycle $C$ and $v:=s(e_1)=r(e_n)$. Since in $C$ there are no edges except $e_i's$, by~(\ref{transgfrt}), $v=r(e_1)(1)=r(e_2)(2)=\dots=r(e_n)(n)$. This shows $v=v(n)$. Writing in the form of ${}^n v=v$, we have that $v$ is a periodic element.  

If the graph $C$ has an exit, then starting from $v=s(e_1)$ and applying the relation~(\ref{transgfrt}) along the line $e_1e_2\dots e_n$, we obtain extra vertices corresponding to exit edges in $C$, i.e., $v={}^n v+ z$, where $z\in T_E$. Thus ${}^n v < v$. Infact, a graph has a cycle without [with] exit if there is a vertex 
$v\in T_E$ such that ${}^nv=v$ [${}^n v < v$] (see~\cite{hazli}).  
\end{example}

Throughout the article, we simultaneously use $v\in E^0$ as a vertex of the graph $E$, as an element of the algebra $L_\K(E)$ and the element $v=v(0)$ in the monoid $T_E$, as the meaning will be clear from the context. For a subset $H\subseteq E^0$, the ideal it generates in $L_\K(E)$ is denoted by $I(H)$, whereas the $\mathbb Z$-order-ideal it generates in $T_E$ is denoted by $\langle H \rangle $.


For the next statement, recall the setting of Lemma~\ref{gdeppok} and the correspondence~(\ref{gfhfgdgdh3})

\begin{lemma} \label{qmiso} Let $E$ be a row-finite graph and $H\subseteq K$ be hereditary saturated subsets of $E^0$.  There is a natural inclusion $i:T_H \rightarrow T_E$ with  $i(T_H)=\langle H \rangle$. Furthermore, there is a natural isomorphism of $\mathbb Z$-monoids 
\begin{align}\label{henhgdfgf}
  T_{K/H} &\longrightarrow T_{K}/T_H,\\
v&\longmapsto v.  \notag
\end{align}
Consequently there is a one-to-one correspondence between ideals of $T_E/ \langle H  \rangle$ and $T_{E/H}$ via 
\[  \langle K \rangle / \langle H \rangle \cong  T_K/T_H \cong T_{K/H} \cong    \langle K/H  \rangle.\]
\end{lemma}
\begin{proof}
Consider the map $i:T_H \rightarrow T_E$ defined on generators by $v\mapsto v$. Since $H$ is hereditary and saturated, the edges emitting from $v$ in $H$ and $E$ are the same, and thus the map $i$ is indeed well-defined. If $i(a)=i(b)$ for $a,b\in T_H$, then an application of Confluence Lemma~\ref{confuu} shows that there is an $c\in T_E$ such that $a=b=c$ in $T_E$. However, since $a,b\in T_H$, all the vertices appearing in the presentations of $a$ and $b$ are in $H$. Thus any transformation on $a$ and $b$ in Confluence Lemma gives  an element in $T_H$ again. Thus $c\in T_H$ and  $i$ is injective. 

To establish~(\ref{henhgdfgf}), define $\phi:T_{K/H} \rightarrow T_K/T_H;v \mapsto v$. This map is well-defined. For consider  the defining relations 
\begin{equation}\label{gfmohtag}
v=\sum_{\{e\in s^{-1}(v)\mid r(e)\not \in H\}} r(e)
\end{equation}
in $T_{K/H}$. In $T_K$, we  have $v=\sum_{\{e\in s^{-1}(v)\mid r(e)\not \in H\}} r(e)+ \sum_{\{e\in s^{-1}(v)\mid r(e) \in H\}} r(e)$. Since the last sum is in $T_H$, we conclude that  (\ref{gfmohtag}) holds in the monoid $T_K/T_H$. 

Now define a map $\psi: T_K \rightarrow T_{K/H}$ on generators by $\psi(v)=v$ if $v\in K\backslash H$ and zero otherwise.  It is routine as above to show this is a well-defined map which factors through $T_H$ and thus induces the homomorphism $\psi:T_K/T_H\rightarrow T_{K/H}$. Since the maps are identity on the generators, $\psi\phi=\id$ and $\phi\psi=\id$, indicating the maps are isomorphism of $\mathbb Z$-monoids. The rest of the lemma follows similarly. 
\end{proof}

We will be using the following Lemma showing a vertex $w\in T_E$ is periodic if and only if it can be represented as a sum of vertices (with shifts) on cycles with no exits. A variant  of the proof given here was used in \cite{hazli,lia}, but we provide a proof for the convenient for the reader.

 \begin{thm} \label{comemaineme} 
Let $E$ be a row-finite directed graph and $T_E$ its associated talented monoid. Then  a vertex $w \in T_E$ is periodic if and only if $w=\sum {}^{i_k} v_k$, where $i_k\in \mathbb Z$ and $v_k$ are vertices on cycles with no exits. 
\end{thm}
\begin{proof}
Suppose $a_i\in T_E$ are periodic elements with  ${}^{p_i}a_i=a_i$, where $0\not= p_i\in \mathbb N$ and $1\leq i \leq n$. Then it is easy to see that ${}^p a=a$, where $a=\sum_{i=1}^n a_i$ and $p=p_1p_2\dots p_n$. Now suppose $w=\sum {}^{i_k} v_k$, where $v_k$ are on cycles without exits. Being on a cycle without exit, each of $v_k$ are periodic, 
thus $w$ is periodic. This gives one direction of the theorem. 

Suppose now a vertex $w\in T_E$ is periodic, i.e., there is an $0\not= l\in \mathbb N$ such that ${}^lw=w$. By the Confluence Lemma~\ref{confuu}, there is a $c\in F_{\overline E}$ such that  $w \rightarrow c$ and $w(l) \rightarrow c$. Write $c=u_1 (m_1)+\dots + u_q(m_q)$.   Since the graph $\overline E$ is a stationary, namely the graph repeats going from level $i$ to level $i+1$, and 
$w(l) \longrightarrow u_1(m_1)+\dots +u_q (m_q),$
then 
$w \longrightarrow u_1(m_1-l)+\dots +u_q (m_q-l)$. Consequently 
 \begin{equation}\label{subgdtgee1}
 u_1(m_1-l)+\dots + u_q (m_q-l) \longrightarrow u_1 (m_1)+\dots + u_q(m_q).
 \end{equation}

 Since the number of generators on the right and the left hand side of (\ref{subgdtgee1}) are the same and the relation $\rightarrow$ would only increase the number of generators if there is  more than one edge emitting from a vertex, it follows that there is only one edge emitting from each vertex in the list 
 $A=\{u_1(m_1-l), \dots , u_q (m_q-l) \}$ and their subsequent vertices until the edges reach the list $B=\{u_1(m_1), \dots , u_q(m_q)\}$. Thus we have a bijection $\rho: A\rightarrow B$. Consequently,  there is an $t\in \mathbb N$ such that $\rho^t=1$. Thus we have $u_i(m_i) \rightarrow u_i (m_i+tl)$ for all elements of $A$. Going back to the graph $E$, this means there is a path with no bifurcation from $u_i$ to itself, namely there is a cycle with no exit based at $u_i$. Since $w=c=u_1 (m_1)+\dots + u_q(m_q)$, the proof is complete. 
 \end{proof}
 
 We are now in a position to provide a proof for Lemma~\ref{hnhngff}.
\begin{proof}[Proof of Lemma \ref{hnhngff}]\label{jumpshere}
Consider the quotient graph $E/H$. Since $E$ has no sinks, by Lemma~\ref{helplemm}(1), $E/H$ has no sinks either. Furthermore, $E/H$ consists of disjoint cycles (otherwise $E$ has does have non-disjoint cycles).

Let $I=\langle H \rangle $ be the $\Gamma$-order-ideal of $T_E$ generated by vertices on the cycles with no exits. Note that $H=I\cap E^0$ is the hereditary and saturated subset generated by vertices on the cycles with no exits. By Lemma~\ref{longpath}, there is an $l$ such that any path of length $l$ starting at $v\in H$ ends at some cycle without exit (the generating set of $H$). Thus in $T_E$, starting from $v$ and applying the $l$ consecutive transformation~\ref{transgfrt}, we get that $v$ can be written as sum of vertices (with shifts) on cycles without exits.  By Theorem~\ref{comemaineme}, $v\in I$ is a periodic element. Since $I$ is generated by all these periodic elements,  $I$ is cyclic.

Consider the maximal chain of cycles $C_1\Rightarrow C_2\Rightarrow \cdots \Rightarrow C_{i-1} \Rightarrow C_{i}$ in $E$. If $i=1$, then $E$ is multi-headed comet and $E/H$ is an empty set. Otherwise, since $C_{i-1}$ is a cycle with an exit in $E$, vertices on $C_{i-1}$ are not cyclic, and thus they don't belong to the cyclic ideal $I$. Consequently they don't belong to $H$. Therefore the cycle $C_{i-1}$ is in the quotient graph $E/H$. An easy argument now shows that $C_1\Rightarrow C_2\Rightarrow \cdots \Rightarrow C_{i-1}$ has to be a maximal chain of cycles in $E/H$.  
\end{proof}

\begin{example}
Let $E$ be the graph. 
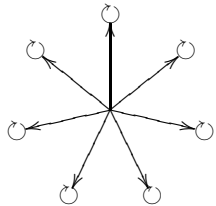
\begin{equation*} 
\xy 
/r3pc/: {\xypolygon7{~<<{}~>{}{\circlearrowright}}}
\endxy
\end{equation*}
\smallskip 

Then the talented monoid $T_E$ is cyclic (Theorem~\ref{comemaineme}). Furthermore, if $z_i$, $1\leq i \leq 7$ denote the vertices of the loops then the order-ideals $\langle z_i\rangle$ are cyclic and minimal and by Lemma~\ref{rtfdm},  $ 0 \subset  \langle z_1 \rangle \subset \langle z_1,z_2 \rangle \subset \dots \subset \langle z_1, \dots z_7 \rangle=T_E$ is a composition series of cyclic type for $T_E$. 

\end{example}

\begin{example}
Let $E$ be the graph. 
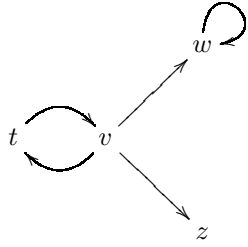
\[
\xymatrix{
&  & w \ar@(u,r) \\
t \ar@/^1pc/[r]& v \ar@/^1pc/[l] \ar[ur] \ar[dr] &   \\
& & z}
\]

\smallskip 
Then the talented monoid $T_E$ has composition series of cyclic and non-comparable mixed type 
$0 \subset \langle w \rangle \subset \langle w,z \rangle  \subset T_E$ and $0 \subset \langle z \rangle \subset \langle w,z \rangle  \subset T_E$.

\end{example}

\begin{example}
Let $E$ be the graph 
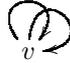
\[
 \xymatrix{
 v \ar@(ul,ur)  \ar@(u,r) 
 }
\]

Then the talented monoid $T_E$ has the composition series $0 \subset T_E$ which is of comparable type. 
\end{example}

Clearly, if $E$ is finite, then $T_E$ has a composition series. In Section~\ref{sec4} we will prove that a graph $E$ consists of disjoint cycles if and only if a composition series of $T_E$ is of cyclic and non-comparable type (Theorem~\ref{mainthm}). In order to do this, we need two theorems. 

\begin{thm}\label{hfghfbggf}
Let $E$ be a finite graph. There is a one-to-one correspondence between cycles with no exits in $E$ and 
the cyclic minimal ideals of $T_E$. Furthermore, the sum of all cyclic minimal ideals, which is the ideal generated by vertices on cycles with no exits, is the largest cyclic ideal of $T_E$. 
\end{thm}
\begin{proof}
Suppose $C$ is a cycle with no exit in $E$ and $v$ is a vertex on $C$. Consider the $\Gamma$-order-ideal $I=\langle v \rangle$. We show that $I$ is a cyclic minimal ideal of $T_E$. We first show that 
\begin{equation}\label{ghnghgdhfh}
\langle v \rangle=\{{}^{i_1} v+\dots+ {}^{i_k} v \mid  i_j \in \mathbb Z\}.
\end{equation}
 If $x\in \langle v \rangle$ then by (\ref{oridealhj}) we have 
$x+t = {}^{i_1} v+\dots+ {}^{i_k} v$ for some $i_j\in \mathbb Z$. Passing to $M_{\overline E}$ and invoking Confluence theorem,  we have an element $c\in F_{\overline E}$ such that $x+t\rightarrow c$ and $v_{i_1} +\dots+ v_{i_k}  \rightarrow c$. Since $v$ is on the cycle $C$ without exit, any transformation of $v$ gives a vertex on the cycle $C$ and thus  $v_{i_1} +\dots+ v_{i_k}  \rightarrow c$ gives that all vertices appearing in the presentation of $c$ are on the cycle $C$. Thus transforming these vertices further if required, we can write $c$ as a sum of vertex $v$ with appropriate shifts. Since $x+t \rightarrow c$, by Lemma~\ref{confuu}(2),  $x$ can also be written as a sum of shifts of $v$. This gives~(\ref{ghnghgdhfh}). Since each element of $I$ is a sum of $v$ (with some shifts), by Theorem~\ref{comemaineme},  $I=\langle v \rangle$ is cyclic.

On the other hand, if $0 \not =J \subseteq I$, then there is an $0 \not = x ={}^{i_1} v+\dots+ {}^{i_k} v \in J$. Since $J$ is a $\Gamma$-order-ideal, ${}^{i_j} v \in J$ and so $v \in J$ which gives that $I=J$. Therefore $I$ is a minimal $\Gamma$-order-ideal. 

Next suppose that $C$ and $D$ are two disjoint cycles with no exit and $v$ and $w$ are vertices on $C$ and $D$, respectively. We show that $\langle v \rangle \not = \langle w \rangle$. Otherwise, $w\in \langle v \rangle$ which by (\ref{ghnghgdhfh}), means $w= {}^{i_1} v+\dots+ {}^{i_k} v$. Again, passing to $M_{\overline E}$ and invoking Confluence theorem, we obtain that $w$ is connected to the cycle $C$, which is a contradiction with the fact that the cycle $D$ has no exit. 

Suppose now that $I$ is a minimal and cyclic ideal of $T_E$. We will show that $I= \langle v \rangle$, where $v$ is a vertex on a cycle with no exit. Since $I\not=0$ is a $\Gamma$-order ideal, there is a vertex $w\in I$, which is periodic (as $I$ is cyclic).  By Theorem~\ref{comemaineme},  $w=v_1(i_1)+\dots v_k(i_k)$, where $v_j$ are all on cycles without exists. Again, $I$ being order-ideal, implies that $v_i$'s all are in $I$. Since $I$ is minimal, it follows that $I=\langle v_i \rangle$. 

For the last part of the lemma, let $I$ be the sum of the cyclic minimal ideals of $T_E$. The first part of the proof shows that  $I$ coincides with the ideal generated by vertices on cycles without exits. Note that the sum of cyclic ideals is cyclic. Thus $I$ is cyclic. Suppose $J$ is the largest cyclic ideal. Thus $I\subseteq J$.  If a vertex $v\in J$, since $v$ is periodic, it can be written as a sum of elements on cycles without exits (Theorem~\ref{comemaineme}). This implies that $v\in I$, hence $I=J$.  
\end{proof}

Recall the notion of a non-comparable ideal from Definition~\ref{noncomdd}. We can identify these ideals in the talented monoid. 

\begin{thm}\label{noncgdf}
Let $E$ be a finite graph. There is a one-to-one correspondence between sinks in $E$ and 
the non-comparable minimal ideals of $T_E$. Furthermore, the sum of all non-comparable minimal ideals is the largest non-comparable ideal of $T_E$. 
\end{thm}
\begin{proof}
Suppose $v$ is a sink in the graph $E$. Consider the $\Gamma$-order-ideal $I=\langle v \rangle$. We show that $I$ is a free monoid generated by ${}^i v$, $i\in \mathbb Z$. If 
\begin{equation}\label{metoo}
{}^{i_1} v+\dots+ {}^{i_k} v = {}^{j_1} v+\dots+ {}^{j_l} v,
\end{equation} then the Confluence Lemma~\ref{confuu} implies that the left and right hand of (\ref{metoo})  should transform to an element $c$ of the free monoid $F_{\overline E}$. However since $v$ is a sink, one can not apply any transformation on $v$ and thus the equality~(\ref{metoo}) already occurs in the free monoid. Thus $l=k$ and $i_s=j_s$, $1\leq s \leq k$, after a suitable permutation. On the other hand, if $x\in I$, then by (\ref{oridealhj}),  $x+a={}^{i_1} v+\dots+ {}^{i_k} v$, for some $i_j\in \mathbb Z$. Another application of Lemma~\ref{confuu}(2) shows that $x$ is a sum of shifts of $v$. 
Thus $I$ is a free monoid generated by ${}^i v$, $i\in \mathbb Z$. Therefore, $I$ is a non-comparable ideal. 
Let $w$ be a sink distinct from $v$. It is easy to show that $\langle v \rangle \not = \langle w \rangle$.

Suppose now that $I$ is a non-comparable $\Gamma$-order-ideal of $T_E$. We first show that any element of $I$ is a sum of some sinks with appropriate shifts. Let $x\in I$. Thus $x=\sum {}^{k_i} v_i$, where $v_i\in E^0$. Since $I$ is a $\Gamma$-order-ideal, $v_i\in I$, for all $i$. We first show that $v_i$ is not connected to any cycle. If $v_i$ connects to a vertex $w$ on a cycle, then  there is an $l\in \mathbb N$, such that ${}^l w \leq v_i$ in $T_E$. Since $I$ is a $\Gamma$-order ideal, $w\in I$. However $w$ is periodic (see Example~\ref{perioddel}), which can't be the case as $I$ is non-comparable. This shows, since $E$ is finite, $v_i$ only connects to sinks. Thus an easy induction shows that $v_i$ can be written as a sum of sinks with appropriate shifts. This proves our claim. Suppose further that $I$ is minimal. Let $x\in I$. So $x=\sum {}^{k_i} v_i$, where all $v_i$ are sinks. Thus $\langle v_1 \rangle \subseteq I$. Since $I$ is minimal, it follows that $I=\langle v_1 \rangle$. Thus minimal non-comparable ideals of $T_E$ are generated by sinks.

For the last part of lemma, let $I$ be the sum of the non-comparable minimal ideals of $T_E$. Note that since elements of a non-comparable ideal is a sum of sinks, the sum of non-comparable ideals is a non-comparable ideal. Thus $I$ is non-comparable. Suppose $J$ is the largest non-comparable ideal. Thus $I\subseteq J$.  If $v\in J$, then $v$ can not be connected to any cycle. For, if $v$ is connected to vertex $w$ on a cycle, then ${}^i w \leq v$ for some $i\in \mathbb N$. Since $J$ is a $\Gamma$-order-ideal, then $w\in J$. However $w$ is a comparable element. Thus $v$ is only connected to sinks. This implies that $v\in I$, and so $J=I$. 
\end{proof}

For the next lemma, we note that if $a\in T_E$ such that  ${}^l a < a$, for some $l\in \mathbb N$, then $a$ can't be periodic. For, if  ${}^k a = a$, for some $k\in \mathbb N$. Then, let $l$ be the least positive integer such that ${}^l a < a$. We then have ${}^l a < a={}^k a$, implying that ${}^{l-k} a < a$. By \cite[Lemma 4.1]{hazli}, $l-k > 0$. Since $l$ was the least positive integer chosen, this leads to a contradiction. 

\begin{cor}\label{gatzia}
Let $E$ be a finite graph. Then the talented monoid $T_E$ is cyclic if and only if $E$ is a multi-headed comet graph. Furthermore,  $T_E$ is simple and cyclic if and only if $E$ is a comet graph. 
\end{cor}
\begin{proof}

If $T_E$ is cyclic, then by  Theorem~\ref{noncgdf}, $E$ has no sink. Furthermore, the graph $E$ has no cycle with an exit. For if $E$ has a cycle $C$ which has an exit at the vertex $v$, then clearly there is an $l\in \mathbb N$, such that ${}^l v < v$, which contradicts the cyclicity of $T_E$. Thus all cycles have no exits and consequently all other vertices are connected to these cycles. Thus $E$ is a multi-headed comet graph.  If $T_E$ is simple, then by Theorems~\ref{hfghfbggf} there is a unique  cycle with no exit in the graph, i.e., $E$ is a comet graph.

Conversely, let $E$ be a multi-headed comet graph. For any vertex $w\in E^0$, an easy induction argument shows that $w\in T_E$ is a sum of some vertices (with appropriate shifts) on the cycles. By Theorem~\ref{comemaineme}, $w$ is periodic.  This gives that $T_E$ is cyclic. If $E$ is a comet, then by Theorems~\ref{hfghfbggf},  there is a minimal cyclic ideal $I=\langle v \rangle$ in $T_E$, where $v$ is a vertex on this cycle. Thus any vertex $w$ ends in this cycle and it follows that $w\in I$. Thus $T_E=I$, i.e., $T_E$ is cyclic and simple. 
\end{proof}

\section{Graphs with disjoint cycles and their talented monoids}\label{sec4}

Whereas multi-headed comet graphs can be characterised via cyclicity of the associated talented monoids (Corollary~\ref{gatzia}),  graphs with disjoint cycles can also be characterised via the composition series of their talented monoids. First we concentrate on graphs with disjoint cycles which have no sinks.

\begin{lemma}\label{mainlemma}

Let $E$ be a finite graph and $T_E$ its talented monoid. Then the following are equivalent.

\begin{enumerate}[\upshape(1)]
\item $E$ consists of disjoint cycles with no sinks;
\smallskip 

\item $T_E$ has a cyclic composition series;

\smallskip 

\item The leading ideal of the upper cyclic series of $T_E$ is $T_E$. 
\end{enumerate}

Furthermore, for such a graph, $l_c(T_E)$ is the maximal length of chain of cycles in $E$. 

\end{lemma}

\begin{proof}
(1) $\Rightarrow$ (2) Suppose $E$ is a graph with no sinks and consists of disjoint cycles. We use an induction on the maximal length of chain of distinct cycles to prove the statement of Lemma. Suppose the maximal length of the chain is $1$. Thus $E$ has to be a multi-headed comet graph. Let $C_i$, $1\leq i \leq k$, be the cycles of the graph, and $v_i\in C_i$ vertices on the cycles. Then by Theorem~\ref{hfghfbggf}, $I_i=\langle v_i \rangle$, $1\leq i \leq k$,  are all the cyclic minimal  ideals of $T_E$. Since each vertex of $E$ can be written as sum of some of the vertices of the cycles, we have $T_E= I_1+\dots+I_k$. Now by Lemma~\ref{rtfdm} we obtain a cyclic composition series for $T_E$. 

 Suppose now that the statement is correct for graphs with maximal length up to $n$. Let $E$ be a graph such that the maximal length of chain  of disjoint cycles is $n+1$. Consider all cycles in $E$ with no exits, name them 
$C_1, \dots, C_k$.  By Theorem~\ref{hfghfbggf}, there are $k$ cyclic minimal corresponding ideals $I_1, \dots, I_k$ in  $T_E$. 
Consider the $\Gamma$-order-ideal $I:=I_1+\dots + I_k$ and the quotient monoid $T_E/I$. Let $H$ be the smallest hereditary and saturated subset containing $\bigcup_{1\leq i \leq k} C_i^0$.  Note that $I= \langle H \rangle$,  and $T_E/I \cong T_{E/  H}$.  Thus by removing all the cycles with no exits from $E$, by Lemma~\ref{hnhngff}, we obtain that the graph $E/H$ has no sinks consisting of disjoint cycles with maximal length of the chain of cycles $n$. By induction, $T_E/I$ therefore has a cyclic composition series. On the other hand by Lemma~\ref{rtfdm}, $I$ has also a cyclic composition series. Now by Lemma~\ref{compso}, $T_E$ has a cyclic composition series. 

(2) $\Rightarrow$ (1)  and (3) $\Rightarrow$ (1) Suppose $T_E$ has a chain 
\[ 0 = I_0 \subset I_{1} \subset \dots \subset I_n=T_E,\] such that $I_{i+1}/I_{i}$, $0 \leq i \leq n-1$, are cyclic monoids. (In (2) the quotients are simple, whereas in the case of (3) they are largest quotient.) Therefore we have a corresponding chain of hereditary and saturated subsets of $E$
\[ \emptyset = H_0 \subset H_{1}  \subset \dots \subset H_n=E^0.\] Since $I_{i+1}/I_i \cong T_{H_{i+1}/H_i}$ (see Lemma~\ref{qmiso})  we conclude that $H_{i+1}/H_{i}$, $0 \leq i \leq n-1$, are multi-headed comet graphs by Corollary~\ref{gatzia}.  If $E$ does have non-disjoint cycles, then there is a vertex $v\in E^0$ which is on two cycles $C$ and $D$. Suppose $v\in H_{i+1}$ but not $H_{i}$, for some $0\leq  i \leq n-1$. Therefore, none of the vertices on $C$ and $D$ also belong to $H_{i}$. This implies that $H_{i+1}/H_{i}$ contains both cycles $C$ and $D$, and therefore it is not a multi-headed comet graph, a contradiction. On the other hand, if $E$ has a sink $v$, then $v\in H_{i+1} \backslash H_{i}$ for some $0\leq i \leq n-1$. Then $v$ is not cyclic in $I_i/I_{i+1}$ which is again a contradiction. This completes the proof. 

(1) $\Rightarrow$ (3) We argue by induction on the length of maximal chain of cycles in $E$. If the length is $1$, then $E$ is a multi-headed comet and by Corollary~\ref{gatzia}, $T_E$ is cyclic. Thus the upper cyclic series is $0 \subset T_E$ with the leading ideal $T_E$ and $l_c(T_E)=1$. Suppose the statement is valid for length $n$ and suppose $E$ is a graph with the length of maximal chains of cycles to be $n+1$. Let $I$ be the largest cyclic ideal of $T_E$ which is the ideal generated by vertices on cycles with no exists in $E$ (Theorem~\ref{hfghfbggf}). Consider $T_E/I \cong T_{E/H}$, where $H$ is the saturated and hereditary subset generated by these vertices. Since the length of maximal chain of cycles of $E/H$ is $n$ (Lemma~\ref{hnhngff}) by induction the leading ideal of $T_E/I$ is $T_E/I$ and $l_c(T_E/I)=n$. Now the third isomorphism theorem for monoids~(\ref{thirdiso}) gives that the leading ideal for $T_E$ is $T_E$ and $l_c(T_E)=n+1$. 

The proof of the last part of the lemma is included in (1) $\Rightarrow$ (3).
\end{proof}

We are in a position to describe a graph consisting of disjoint cycles (possibly having sinks) via its talented monoid.

\begin{thm}\label{mainthm}
Let $E$ be a finite graph, $L_\K(E)$ its associated Leavitt path algebra and $T_E$ its talented monoid.  Then the following are equivalent.

\begin{enumerate}[\upshape(1)]
\item $E$ is a graph with disjoint cycles; 

\item $T_E$ has a composition series of cyclic and non-comparable types; 

\item $L_\K(E)$ has a finite GK-dimension. 
\end{enumerate}

\end{thm}

\begin{proof}

(1) $\Rightarrow$ (2) Suppose $E$ is a graph with distinct cycles. If $E$ has no sinks, then by Lemma~\ref{mainlemma}, $T_E$ has a cyclic composition series.  Otherwise, let $\{v_1,\dots, v_k\}$ be the set of  all the sinks in $E$. By Theorem~\ref{noncgdf}, $I_i=\langle v_i\rangle$, $1\leq i \leq k$,  are non-comparable minimal ideals of $T_E$ and by Lemma~\ref{rtfdm} the series $0 \subset I_1 \subset I_1+ I_2 \subset \dots \subset I_1+I_2+\dots +I_k$ is a composition series of non-comparable type for  $I=I_1+I_2+\dots +I_k$. 

Let $H$ be the smallest hereditary and saturated subset containing $\{v_1,\dots, v_k\}$. Then $\langle H \rangle =I$. Consider the quotient graph $E/H$. If $E/H$ is not empty, i.e., $I \not = T_E$ and $E$ has cycles, then by Lemma~\ref{helplemm}, $E/H$ is a graph with no sink consisting of disjoint cycles. Therefore, by Lemma~\ref{mainlemma}, $T_{E/H}\cong T_E/I$ has a cyclic composition series. Since $I$ has a non-comparable composition series, by Lemma~\ref{compso}, $T_E$ has a mixed composition series. 

(2) $\Rightarrow$ (1) The argument is similar to the corresponding proof of Lemma~\ref{mainlemma}. Suppose $T_E$ has a chain of ideals
$0 = I_0 \subset I_{1} \subset \dots \subset I_n=T_E$ and its corresponding chain of hereditary and saturated subsets of $E$, 
$ \emptyset = H_0 \subset H_{1}  \subset \dots \subset H_n=E^0$, 
such that $I_{i+1}/I_{i} \cong T_{H_{i+1}/H_i}$, $0 \leq i \leq n-1$, are either cyclic or non-comparable. We only need to show that $E$ does not contain vertices which are base of more than one cycles. Suppose $v\in E^0$ is a vertex  on two cycles $C$ and $D$. Then there is an $0\leq  i \leq n-1$ such that  $v\in H_{i+1}$ but not $H_{i}$. Therefore, none of the vertices on $C$ and $D$ also belong to $H_{i}$. This implies that $H_{i+1}/H_{i}$ contains both cycles $C$ and $D$. Therefore,  there is an $j \in \mathbb N$ such that ${}^j v < v$ in $ T_{H_{i+1}/H_i}$. But these monoids are either cyclic or non-comparable. This is a contradiction and thus $E$ consists of disjoint cycles.

(1) $\Leftrightarrow$ (3) This is Theorem~\ref{zelgk}. 

\end{proof}

\begin{example}\label{ggdgdgdgd}
Consider the graphs with disjoint cycles below. 
\medskip 

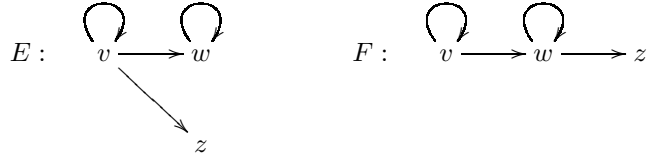
\begin{equation*}
{\def\labelstyle{\displaystyle}
E : \quad \,\,  \xymatrix{
v  \ar@(ul,ur) \ar[r]\ar[dr] &  w \ar@(ul,ur) \\
& z
}}
\qquad \quad \quad \quad
{\def\labelstyle{\displaystyle}
F : \quad \,\,  \xymatrix{
v  \ar@(ul,ur) \ar[r] &  w \ar@(ul,ur)  \ar[r]&  z\\
}}
\end{equation*}
By~\cite{zel} we have $\GKdim L_\K(E)=3$ and $\GKdim L_\K(F)=4$. Note that $T_E$ and $T_F$ can't be isomorphic as $T_E$ has a periodic element (${}^1 w=w$), whereas $T_F$ has no periodic elements (otherwise it has to have a cycle with no exit). However, we have the following composition series for $T_E$ and $T_F$, which are equivalent: 
\begin{align*}
\langle  z \rangle  \subset \langle  z, w \rangle  \subset T_E,\\
\langle  z \rangle  \subset \langle  z, w \rangle  \subset T_F.
\end{align*}

\end{example} 

Although non-isomorphic monoids can have equivalent composition series, we show next that one can determine the Gelfand-Kirillov dimension of the Leavitt path algebra $L_\K(E)$ from the talented monoid $T_E$. For the next theorem, recall the notion of leading ideal of the upper cyclic series and its length $l_c(T_E)$ from Definition~\ref{gdhfyrhfu}. 

\begin{thm}\label{gkdiml}
Let $E$ be a finite graph with disjoint cycles, $L_\K(E)$ its associated Leavitt path algebra and $T_E$ its talented monoid.  Let $S$ be the largest non-comparable 
$\Gamma$-order-ideal  and $I$ the leading ideal of the upper cyclic series of $T_E$. Let $d_1=l_c(T_E/S)$ and $d_2=l_c(T_E/S+I)$.   Then 
\begin{equation}\label{gkformula}
\GKdim L_\K(E)= \begin{cases}
2d_1 &\text{if $d_1=d_2$;}\\
2d_1-1 &\text{if $d_1 \not = d_2$.}
\end{cases}
\end{equation}
\end{thm}
\begin{proof}
We show that  $d_1$ is the maximal length of  chain of cycles in $E$ and $d_2$ is the maximal length of chain of cycles ending at a sink  in $E$. 
Let $H$ be the hereditary and saturated subset generated by all sinks in $E$. By Lemma~\ref{noncgdf},  $S=\langle H \rangle$ and $T_E/S\cong T_{E/H}$. By Lemma ~\ref{helplemm}, the quotient graph $E/H$ consists of disjoint cycles with no sinks and there is a one-to-one correspondence between the cycles  in $E$ and $E/H$. Thus any (maximal) chain of cycles in $E$ does also appear in $E/H$ and vice versa.   Since $E/H$ has no sinks, by Lemma~\ref{mainlemma},  $d_1=l_c(T_E/S)$ is the maximal length of  chain of cycles in $E/H$  and consequently in $E$.

Now let $I$ be the leading ideal of the upper cyclic series of $T_E$. Thus 
$T_E$ has a series 
\[ 0 = I_0 \subset I_{1} \subset \dots \subset I_n=I,\] such that $I_{i+1}/I_{i}$, $0 \leq i \leq n-1$, are largest cyclic ideals of $T_E/I_i$. Therefore we have a corresponding chain of hereditary and saturated subsets of $E$
\begin{equation}\label{tegdgff} 
\emptyset = G_0 \subset G_{1}  \subset \dots \subset G_n=G.
\end{equation}
Since $I$ is the leading ideal,  $T_E/I\cong T_{E/G}$ has no non-zero periodic element. 
Therefore, $E/G$ has no cycle without exit. Since $E/G$ is a graph with disjoint cycles, we conclude that all maximal chain of cycles in $E/G$ ends at  a sink. We show that these are exactly the maximal chain of cycles with an exit in $E$.  
We first show that there is one-to-one correspondence between sinks in $E/G$ and $E$ via the morphism $\phi: E/G\rightarrow E$. 
We argue on the chain of $G_i$'s in (\ref{tegdgff}). By the definition of the upper cyclic series, since $I_1$ is the largest cyclic ideal of $T_E$, by Theorem~\ref{hfghfbggf}, $G_1$ is the hereditary and saturated subset generated by vertices of all cycles with no exits. Thus by Lemma~\ref{helplemm}(4),
there is a one-to-one correspondence between sinks under the map $\phi:E/G_1\rightarrow E$.

The graph $E/G_1$ consists of disjoint cycles and $T_E/I_1\cong T_{E/{G_1}}$. Since $I_2/I_1$ is the largest cyclic ideal of $T_E/I_1$, by Lemmas~\ref{gdeppok} and \ref{qmiso},  
the subset $G_2$ corresponds to hereditary and saturated subset in the graph $E/G_1$ generated by all cycles without exists in $E/G_1$. Therefore by Lemma~\ref{helplemm}(4), there is a one-to-one correspondence between sinks under the map $\phi:(E/G_1)\big /(G_2/G_1) \rightarrow E/G_1$. But $(E/G_1)\big /(G_2/G_1)=E/G_2$ (see Lemma~\ref{gdeppok}). Therefore there is a one-to-one correspondence between sinks under the map $E/G_2\rightarrow E$. Continuing this way, we obtain one-to-one correspondence between sinks in $E/G$ and $E$.

Next we show any maximal chain of cycles ending at a sink in $E$ does appear in $E/G$. Suppose $C_1\Rightarrow C_2 \Rightarrow \dots \Rightarrow C_k$ is a maximal chain ending at a sink $v$  in $E$. Thus there is an exit from the cycle $C_k$ which ends in $v$. Since there is a one-to-one correspondence between sinks in $E$ and $E/G$,  $v$ is also a sink in $E/G$.  We conclude that none of the $C_i$'s are in $G$, otherwise $v \in G$ as $G$ is hereditary which is not the case.  Therefore the chain $C_1\Rightarrow C_2 \Rightarrow \dots \Rightarrow C_k$ appears in $E/G$ as well. Thus the maximal chains of cycles ending at a sinks appear in both $E$ and $E/G$. 

To obtain the length of these chains, we remove all the sinks from $E/G$, which becomes the graph $E/\langle H,G\rangle$, where $H$ was the hereditary and saturated subset generated by sinks in $E/G$ (which is the same as in $E$).  By Lemma~\ref{helplemm}(3),  the same chain of cycles appear in the graph $E/\langle H,G\rangle$ as well. Note that $T_{E/\langle H,G\rangle}\cong T_E/S+I$. 
Now we obtain the length of the longest chains via Lemma~\ref{mainlemma}, which is $l_c(T_E/S+I)$.  Thus $d_2=l_c(T_E/S+I)$ is the maximal length of chain of cycles ending at a sink in $E$ as well. 
The formula (\ref{gkformula}) now follows from Theorem~\ref{zelgk}. 
\end{proof}

\begin{cor}\label{hfgfhfdd}
Let $E$ and $F$ be finite graphs. If there is a $\mathbb Z$-isomorphism of order groups $K_0^{\gr}(L_\K(E)) \cong K_0^{\gr}(L_\K(F))$,  then $\GKdim L_\K(E) = \GKdim L_\K(F)$. 
\end{cor}
\begin{proof}
Since $T_E$ is the positive cone of  the graded Grothendieck group $K_0^{\gr}(L_\K(E))$~\cite{arali}, we have $T_E\cong T_F$ as $\mathbb Z$-monoids if and only if there is a $\mathbb Z$-isomorphism $K_0^{\gr}(L_\K(E)) \cong K_0^{\gr}(L_\K(F))$, as order groups. 

If $\GKdim L_\K(E)<\infty$ then $E$ is a graph with disjoint cycles and sinks (Theorem~\ref{zelgk}) and  thus by Theorem \ref{mainthm}, $T_E$ has a composition series. Thus $T_F$ has a composition series and again Theorem \ref{mainthm} implies that $\GKdim L_\K(F)<\infty$. Now, Theorem \ref{gkdiml} gives that $\GKdim L_\K(E) = \GKdim L_\K(F)$. 

On the other hand if $\GKdim L_\K(E)= \infty$, then the graph $E$ is not a graph with disjoint cycles. Theorem~\ref{mainthm}  implies the same for the graph $F$ and consequently, $\GKdim L_\K(F)= \infty$.
\end{proof}



 \end{document}